\documentclass[12pt]{article}
\usepackage{amsmath}
\usepackage{amssymb}
\usepackage{amscd}
\usepackage{amsthm}

\theoremstyle{plain}
\newtheorem{Thm}{Theorem}[section]

\newtheorem{Prop}[Thm]{Proposition}
\newtheorem{Cor}[Thm]{Corollary}
\newtheorem{Lem}[Thm]{Lemma}

\theoremstyle{definition}

\newtheorem{Expl}[Thm]{Example}

\numberwithin{equation}{section}

\title{Semipositivity theorem for reducible algebraic fiber spaces}
\author{Yujiro Kawamata}

\begin{document}

\maketitle

\begin{abstract}
We shall prove an extension of the semipositivity theorem for the case of 
reducible algebraic fiber spaces.
\end{abstract}

\section{Introduction}

The canonical divisor is an important invariant of an
algebraic variety and reflects its geometric properties
as an algebro-geometric curvature.
When we compare canonical divisors of related varieties in many
different situations, then we always find certain inequalities.
This phenomena is not symmetric and called the positivity of 
the canonical divisor.

This paper is concerned with the positivity of canonical divisors 
for algebraic fiber 
spaces whose fibers are not necessarily irreducible.
The irreducible case was proved and used for the proof of a 
subadditivity theorem of the Kodaira dimension (\cite {CA}) and the 
subadjunction theorem for higher codimensional subvarieies 
(\cite{subadj}).
This kind of positivity theorem is one of the main geometric applications 
of the Hodge theory.
There is a metric version of the positivity theorem by 
Berndtson and Paun (\cite{BP}).

We consider in this paper an extended case where the total space
of an algebraic fiber space is non-normal, and prove the existence of a 
variations of mixed Hodge structures adapted to our situation by 
extending an argument in an article \cite{CA} which considered 
irreducible algebraic fiber spaces.

We fix the terminology.
A reduced equi-dimensional algebraic scheme $X$ is said to be a 
{\em simple normal crossing variety}  
if it has smooth irreducible components and only normal crossing singularities.
A {\em closed stratum} of $X$ is 
an irreducible component of the intersection of some of the irreducible
components of $X$. 
Thus an irreducible component of $X$ is a closed stratum, and
a closed stratum is a smooth variety.

Let $B$ be a reduced Cartier divisor on a simple normal crossing variety $X$.
The pair $(X,B)$ is said to be a {\em simple normal crossing pair} if the
following conditions are satisfied:
\begin{itemize}
\item $B$ does not contain any closed stratum of $X$.

\item Let $Z$ be a closed stratum of $X$.
Then the union $B_Z$ of the intersection $B \cap Z$ and 
all the closed strata properly cotained in $Z$ 
is a simple normal crossing divisor on $Z$.
\end{itemize}

A {\em closed stratum of the pair} $(X,B)$ is either a closed stratum $Z$ 
of $X$ or an irreducible component of the intersection of some of the 
irreducible components of the divisor $B_Z$.

A locally free sheaf $F$ on a complete variety $Y$ is said to be 
{\em numerically semi-positive} 
if the tautological invertible sheaf $\mathcal{O}_{\mathbf{P}(F)}(1)$ on 
the projectivized bundle $\mathbf{P}(F)$ is nef.
It is a weaker condition than the existence of semi-positive metric,
which is still weaker than the condition that the sheaf is generated by
global sections.

We work over $\mathbf{C}$.
We state the main theorem of this paper:

\begin{Thm}\label{main}
Let $X$ be a projective simple normal crossing variety,
$Y$ a smooth projective irreducible variety, 
and $f: X \to Y$ a projective surjective morphism.
Let $B$ be a reduced Cartier divisor on $X$ such that $(X,B)$ is a 
simple normal crossing pair, and 
$C$ a simple normal crossing divisor on $Y$.
Let $q$ be a non-negative integer.
Assume the following conditions:
\begin{itemize}
\item Let $Z$ be a closed stratum of the pair $(X,B)$.
Then the induced morphism $f \vert_Z: Z \to Y$ is surjective, 
has connected fibers, and is smooth over $Y \setminus C$.

\item The local monodromies of the local system
\[
R^{d(Z/Y)+q}(f \vert_{Z^o})_*\mathbf{Z}_{Z^o}
\]
around the branches of $C$ are unipotent, where 
$d(Z/Y) = \dim Z - \dim Y$ denotes the fiber dimension of $f \vert_Z$, 
and $Z^o = Z \setminus (f \vert_Z)^{-1}(C)$.
\end{itemize}
Then the sheaf $R^qf_*\omega_{X/Y}(B)$ is locally free over $Y$ and 
is numerically semipositive, where $\omega_{X/Y}(B)$ is the relative 
canonical sheaf with logarithmic poles along $B$.
\end{Thm}

This result was proved if $X$ is irreducible in \cite{Fujino} after 
preceding results by \cite{Fujita} in the case $q=0$, $B = 0$ and 
$\dim Y = 1$, by \cite{CA} in the case $q=0$ and $B = 0$, 
and by \cite{Kollar} and \cite{Nakayama} in the case $B = 0$.

The main result can be extended without change of the proof to the case where 
$X$ is a K\"ahler space, because we use only the fact that the cohomology 
groups of compact K\"ahler manifolds underly Hodge structures.
The polarizations are defined over $\mathbf{R}$ instead of $\mathbf{Q}$, so we
have to consider variations of mixed Hodge $\mathbf{R}$-structures in this 
case.

We use the strategy of \cite{CA} for the proof.
We construct a variation of mixed Hodge structures 
on the open part $Y \setminus C$, 
and prove that the sheaf $R^qf_*\omega_{X/Y}(B)$ coincides with 
the canonical extensions of the graded piece of the highest degree 
with respect to the Hodge filtration.
Then the semipositivity theorem follows from a result in \cite{CA}, which is 
a consequence of Griffiths' result (\cite{Griffiths}) 
on the semi-positivity of the metric
of the Hodge bundles over the open part $Y \setminus C$ together with the 
considerations at the boundary $C$ using results of Schmid (\cite{Schmid}).

The local monodromies are always quasi-unipotent.
Therefore the assumption on the local monodromies in the theorem 
is easily achieved if 
we replace $Y$ by a Kummer covering and change $X$ accordingly.
This is the unipotent reduction theorem in \cite{CA}.
We can also put it that the semipositivity theorem holds for orbibundles on
the orbifold above $Y$.

If the fiber space is irreducible, then we have a weak semistable 
reduction theorem (\cite{AK}), and we can analyse the mixed Hodge structures on
the degenerate fibers more explicitly (\cite{AF}).
But we do not know whether there is an extension of the weak semistable 
reduction theorem to the case of reducible algebraic fiber spaces.

Although our result is stated in a numerical language,
the proof of the semipositivity theorem says something stronger 
than the purely numerical one; namely
the Hodge bundle carries a semipositive singular hermitian metric 
whose singularities have no contributions to the
curvature form.

We note that the underlying local system of the 
variation of mixed Hodge structures does not come from a 
locally trivial family of topological spaces, e.g., fibers of $f$ over 
$Y \setminus C$, but something virtually corresponding to their duals.

The contents of this paper is as follows.
We recall the terminology on the cohomological mixed Hodge complexes and
variations of mixed Hodge structures in \S 2.
We start with the construction of the mixed Hodge structures 
on the general fibers in \S 3, 
and then make it relative over the base in \S 4.
The main result is proved in \S 5.

We thank Valery Alexeev and Christopher Hacon 
whose question at MSRI Berkeley on whether the 
statement of the theorem holds was the starting point of the project.
We also thank MSRI for the excellent working condition.


\section{Preliminaries}

\subsection{Cohomological mixed Hodge complexes}

We recall the definition of cohomological mixed Hodge complexes in 
\cite{D}.

(1) \cite{D}~8.1.1.
A {\em Hodge $\mathbf{Q}$-complex} of weight $n$ 
\[
H = (H_{\mathbf{Q}}, (H_{\mathbf{C}},F), \alpha)
\]
consists of the following data:

\begin{itemize}
\item a complex of $\mathbf{Q}$-modules $H_{\mathbf{Q}} \in 
\text{Ob }D^+(\mathbf{Q})$ which is bounded from below and such that 
the cohomologies $H^k(H_{\mathbf{Q}})$ are $\mathbf{Q}$-modules 
of finite type for all $k$;

\item a filtered complex $(H_{\mathbf{C}},F) \in 
\text{Ob }D^+F(\mathbf{C})$ which is bounded from below
with an isomorphism 
$\alpha: H_{\mathbf{C}} \cong H_{\mathbf{Q}} \otimes \mathbf{C}$ in 
$D^+(\mathbf{C})$;
\end{itemize}
such that the spectral sequence associated to the filtration $F$ 
\[
E_1^{p,q} = H^{p+q}(\text{Gr}_F^p(H_{\mathbf{C}})) \Rightarrow 
H^{p+q}(H_{\mathbf{C}})
\]
degenerate at $E_1$, and 
\[
(H^k(H_{\mathbf{Q}}), (H^k(H_{\mathbf{C}}),F))
\]
with an isomorphism 
$H^k(\alpha): H^k(H_{\mathbf{C}}) \cong H^k(H_{\mathbf{Q}}) \otimes 
\mathbf{C}$
is a Hodge $\mathbf{Q}$-structure of weight $n+k$.

If $H$ is a Hodge $\mathbf{Q}$-complex of weight $n$, then
$H[m]$ is a Hodge $\mathbf{Q}$-complex of weight $m+n$ (\cite{D}~8.1.4).

(2) \cite{D}~8.1.2.
A {\em cohomological Hodge $\mathbf{Q}$-complex} of weight $n$ 
on a topological space $X$ 
\[
H = (H_{\mathbf{Q}}, (H_{\mathbf{C}},F), \alpha)
\]
consists of the following data:

\begin{itemize}
\item a complex $H_{\mathbf{Q}} \in \text{Ob }D^+(X, \mathbf{Q})$ 
of locally constant sheaves on $X$ which is bounded from below;

\item a filtered complex $(H_{\mathbf{C}},F) \in 
\text{Ob }D^+F(X, \mathbf{C})$ of sheaves on $X$ which is bounded from below
with an isomorphism 
$\alpha: H_{\mathbf{C}} \cong H_{\mathbf{Q}} \otimes \mathbf{C}$ in 
$D^+(X, \mathbf{C})$;
\end{itemize}
such that the triple
\[
(R\Gamma(H_{\mathbf{Q}}), R\Gamma((H_{\mathbf{C}},F)), R\Gamma(\alpha))
\]
is a Hodge $\mathbf{Q}$-complex of weight $n$.

(3) \cite{D}~8.1.5.
A {\em mixed Hodge $\mathbf{Q}$-complex} 
\[
H = ((H_{\mathbf{Q}},W), (H_{\mathbf{C}},W,F), \alpha)
\]
consists of the following data:
\begin{itemize}
\item a filtered complex of $\mathbf{Q}$-modules 
$(H_{\mathbf{Q}},W) \in \text{Ob }D^+F(\mathbf{Q})$
which is bounded from below
with an increasing filtration $W$ such that 
the cohomologies $H^k(H_{\mathbf{Q}})$ are $\mathbf{Q}$-modules of finite type
for all $k$;

\item a bifiltered complex $(H_{\mathbf{C}},W,F) \in 
\text{Ob }D^+F_2(\mathbf{C})$ which is bounded from below with an isomorphism 
$\alpha: (H_{\mathbf{C}},W) \cong (H_{\mathbf{Q}},W) \otimes \mathbf{C}$ in 
$D^+F(\mathbf{C})$;
\end{itemize}
such that 
\[
(\text{Gr}_n^W(H_{\mathbf{Q}}), (\text{Gr}_n^W(H_{\mathbf{C}}),F))
\]
with an isomorphism
\[
\text{Gr}_n^W(\alpha): \text{Gr}_n^W(H_{\mathbf{C}})
\cong \text{Gr}_n^W(H_{\mathbf{Q}}) \otimes \mathbf{C}
\]
is a Hodge $\mathbf{Q}$-complex of weight $n$ for any $n$.

(4) \cite{D}~8.1.6.
A {\em cohomological mixed Hodge $\mathbf{Q}$-complex} 
on a topological space $X$ 
\[
H = ((H_{\mathbf{Q}},W), (H_{\mathbf{C}},W,F), \alpha)
\]
consists of the following data:

\begin{itemize}
\item a filtered complex $(H_{\mathbf{Q}},W) \in \text{Ob }D^+F(X, \mathbf{Q})$
which is bounded from below with an increasing filtration $W$ such that 
the $H^k(X, H_{\mathbf{Q}})$ are $\mathbf{Q}$-modules of finite type;

\item a bifiltered complex $(H_{\mathbf{C}},W,F) \in 
\text{Ob }D^+F_2(X, \mathbf{C})$ which is bounded from below with an isomorphism 
$\alpha: (H_{\mathbf{C}},W) \cong (H_{\mathbf{Q}},W) \otimes \mathbf{C}$ in 
$D^+F(X, \mathbf{C})$;
\end{itemize}
such that 
\[
(\text{Gr}_n^W(H_{\mathbf{Q}}), (\text{Gr}_n^W(H_{\mathbf{C}}),F))
\]
with an isomorphism
\[
\text{Gr}_n^W(\alpha): \text{Gr}_n^W(H_{\mathbf{C}})
\cong \text{Gr}_n^W(H_{\mathbf{Q}}) \otimes \mathbf{C}
\]
is a cohomological Hodge $\mathbf{Q}$-complex of weight $n$ for any $n$.

If $X$ is a proper smooth variety and $B$ is a normal crossing divisor, then
\[
((Ri_*\mathbf{Q}_{X \setminus B}, \tau), 
(\Omega_X^{\bullet}(\log B), W, \sigma))
\]
is a cohomological mixed Hodge $\mathbf{Q}$-complex on $X$, where 
$i: X \setminus B \to X$ is an open immersion, $\tau$ is a canonical 
filtration, $W$ is the filtration with respect to the order of log poles,
and $\sigma$ is a stupid filtration (\cite{D}~8.1.8).

\begin{Prop}\cite{D}~8.1.7.
If $H = ((H_{\mathbf{Q}},W), (H_{\mathbf{C}},W,F))$ is a 
cohomological mixed Hodge $\mathbf{Q}$-complex, then 
$R\Gamma H = (R\Gamma (H_{\mathbf{Q}},W), R\Gamma (H_{\mathbf{C}},W,F))$ is 
a mixed Hodge $\mathbf{Q}$-complex.
\end{Prop}

\begin{Thm}\label{degenerate}\cite{D}~8.1.9.
Let $H = ((H_{\mathbf{Q}},W), (H_{\mathbf{C}},W,F), \alpha)$ be a 
mixed Hodge $\mathbf{Q}$-complex.
Then the following hold:

(1) $H^n(H_{\mathbf{Q}}) = ((H^n(H_{\mathbf{Q}}),W[n]), 
(H^n(H_{\mathbf{C}}),W[n],F), H^n(\alpha))$ 
is a mixed Hodge $\mathbf{Q}$-structure.

(2) The spectral sequence defined by the weight filtration
\[
{}_WE_1^{p,q}= H^{p+q}(\text{Gr}_{-p}^W(H_{\mathbf{Q}})) \Rightarrow 
H^{p+q}(H_{\mathbf{Q}})
\]
degenerates at $E_2$.

(3) The spectral sequence defined by the Hodge filtration
\[
{}_FE_1^{p,q}= H^{p+q}(\text{Gr}^p_F(H_{\mathbf{C}})) \Rightarrow 
H^{p+q}(H_{\mathbf{C}})
\]
degenerates at $E_1$.
\end{Thm}


\subsection{Variations of mixed Hodge structures}

A {\em variation of mixed Hodge $\mathbf{Q}$-structures} 
on a complex manifold $Y$
\[
H = ((H_{\mathbf{Q}},W), (\mathcal{H},W,F), \alpha)
\]
consists of the following data:

\begin{itemize}
\item a filtered locally constant sheaf of  
$\mathbf{Q}$-modules of finite rank
$(H_{\mathbf{Q}},W) \in \text{Ob }D^+(Y,\mathbf{Q})$
with an increasing filtration $W$ by locally constant sheaves; 

\item a bifiltered locally free sheaf of coherent 
$\mathcal{O}_Y$-modules
$(\mathcal{H},W,F) \in \text{Ob }D^+F_2(X,\mathcal{O}_Y)$ by 
an increasing filtration $W$ by locally free sheaves
with an isomorphism 
$\alpha: (\mathcal{H},W) \cong (H_{\mathbf{Q}},W) \otimes \mathcal{O}_Y$ in 
$D^+F(X,\mathcal{O}_Y)$, 
and a decreasing filtration $F$ by locally free sheaves 
\end{itemize}
such that the natural connection $\nabla$ on $\mathcal{H}$ induced 
from $\alpha$ satisfies
\[
\nabla: F^p\mathcal{H} \to F^{p-1}\mathcal{H} \otimes \Omega_Y^1
\]
and, for each point $y \in Y$, 
$H_y = ((H_{\mathbf{Q},y},W), (\mathcal{H} \otimes \mathcal{O}_y, W,F), 
\alpha)$ is a
mixed Hodge $\mathbf{Q}$-structure.

A variation of mixed Hodge $\mathbf{Q}$-structures $H$ 
is said to be {\em graded polarizable} if $\text{Gr}^W_q(H)$ are
polarized variations of Hodge $\mathbf{Q}$-structures 
for all $q$ (\cite{SZ}).

Let $H = (H_{\mathbf{Q}}, (\mathcal{H},F))$ be a polarized variation 
of (pure) Hodge $\mathbf{Q}$-structures. 
By \cite{Griffiths}~Theorem~5.2, if $d$ is the largest number 
such that $F^d\mathcal{H} \ne 0$, then the 
curvature form $\Theta$ of the Hodge bundle $F^d\mathcal{H}$ 
is given by the formula
\[
(\Theta e, e') = (\sigma(e), \sigma(e'))
\]
where $\sigma: F^d\mathcal{H} \to F^{d-1}\mathcal{H}/F^d\mathcal{H} 
\otimes \Omega^1_Y$ 
is induced from the connection $\nabla$ of the flat bundle $\mathcal{H}$ and 
the inner product is given by the polarization. 
In particular, the curvature $\Theta$ is semipositive differential form of 
type $(1,1)$.

Let $H$ be a variation of mixed Hodge $\mathbf{Q}$-structures over $Y$, and 
assume that $Y$ is embedded into a complex manifold $\tilde Y$
such that the complement $C = \tilde Y \setminus Y$ is a
normal crossing divisor.
Assuming that the local monodromies of the local system $H_{\mathbf{Q}}$ 
around the branches of $C$ are unipotent, 
we define a {\em canonical extension} of $H$ as follows.
Let $y \in C$ be an arbitrary point at the boundary, 
let $N_i$ be the logarithm of the local monodromies around 
the branches of $C$ around $y$,
and let $t_i$ be the local coordinates defining the branches.
Then the canonical extension $\tilde{\mathcal{H}}$ is 
defined as a locally free sheaf on $\tilde Y$ 
which is generated by sections of the form 
$\tilde v= \text{exp}(-\sum_i t_iN_i/2\pi \sqrt{-1})(v)$ near $y$, 
where the $v$ are multi-valued flat sections of $H_{\mathbf{Q}}$.
We note that the monodromy actions are canceled and the $\tilde v$
are single-valued holomorphic sections of $\mathcal{H}$ outside the 
boundary divisors.

The Hodge filtration $F$ of $\mathcal{H}$ extends
to $\tilde{\mathcal{H}}$ such that $\text{Gr}_F(\tilde{\mathcal{H}})$ is still
a locally free sheaf on $\tilde Y$.
Indeed this is a consequence of the nilpotent orbit theorem (\cite{Schmid}) 
when $H$ is a variation of pure Hodge structures, 
and the general case follows immediately from this.


\section{Absolute case}

In this subsection we do the calculation on the fibers of $f$ 
over $Y \setminus C$.
In other words, we consider the case where $Y$ is a point.
The purpose is also to explain the idea of the proof without too much notation.

Let $(X,B)$ be a projective simple normal crossing pair.
We denote by $X^{[q]}$ the disjoint union of all the
closed strata of $X$ which are of codimension $q$.
For example, $X^{[0]}$ is the normalization of $X$, 
$X^{[1]}$ is the normalization of the double locus of $X$,
and so on.
We denote by $B^{[q]}$ the disjoint union of the $B_Z$
for all the irreducible components $Z$ of $X^{[q]}$.
It is a simple normal crossing divisor on $X^{[q]}$.
Let $i_q: X^{[q]} \setminus B^{[q]} \to X^{[q]}$ 
be the open immersions. 

In the following, we identify sheaves with their direct 
images under finite morphisms by the abuse of notation.

First we define the {\em logarithmic De Rham complex} 
$\tilde{\Omega}^{\bullet}_X(\log B)$ for the pair $(X,B)$
as the first term of the 
following exact sequence given by the residue homomorphisms
\[
\begin{split}
&0 \to \tilde{\Omega}^{\bullet}_X(\log B)
\to \Omega^{\bullet}_{X^{[0]}}(\log B^{[0]})
\to \Omega^{\bullet}_{X^{[1]}}(\log B^{[1]})[-1] \\
&\to \Omega^{\bullet}_{X^{[2]}}(\log B^{[2]})[-2]
\to \dots
\end{split}
\]
For example
\[
\tilde{\Omega}^{\dim X}_X(\log B) \cong \omega_X(B)
\]
is an invertible sheaf, but 
\[
\tilde{\Omega}^0_X(\log B) \cong \mathcal{O}_{X^{[0]}}
\]
is the structure sheaf of the normalization of $X$. 

If $X$ is embedded into a 
smooth variety $V$ as a simple normal crossing divisor, and 
if $B$ coincides with a restriction $B = W \cap X$ of another 
simple normal crossing 
divisor $W$ on $V$ such that $X \cup W$ is also a simple normal crossing 
divisor, then there is an exact sequence given by the residue homomorphism
\[
0 \to \Omega^{\bullet}_V(\log W) \to \Omega^{\bullet}_V(\log (X+W)) 
\to \tilde{\Omega}^{\bullet}_X(\log B)[-1] \to 0.
\]

Our complex is differnt from the complex $\bar{\Omega}_X^{\bullet}(\log B)$
considered by Du Bois \cite{DB}, which is a single complex 
associated to the following double complex coming from the Mayor-Vietoris 
exact sequence:
\[
\bar{\Omega}_X^{\bullet}(\log B) = 
\{ \Omega^{\bullet}_{X^{[0]}}(\log B^{[0]}) \to
\Omega^{\bullet}_{X^{[1]}}(\log B^{[1]}) 
\to \Omega^{\bullet}_{X^{[2]}}(\log B^{[2]})
\to \dots\}
\]
where the arrows are the restriction homomorphisms.
We note that both these complexes are different from the complex of K\"ahler 
differentials.
We have $\bar{\Omega}_X^0 \cong \mathcal{O}_X$, and 
\[
Ri_*\mathbf{C}_{X \setminus B} \cong \bar{\Omega}_X^{\bullet}(\log B)
\]
for the open immersion $i: X \setminus B \to X$.

On the other hand, there is no obvious topological 
space which corresponds to the complex
$\tilde{\Omega}^{\bullet}_X(\log B)$ in a similar way as in the case of 
Du Bois complex.
In order to define an object on the $\mathbf{Q}$-level,
we consider closed strata $Z$ and $Z'$ of $X$ such that $Z'$ is contained in 
$Z$ as a prime divisor.
Let $B'_Z$ be the union of all the irreducible components of $B_Z$ 
except $Z'$. 
Then we have $B_{Z'}=B'_Z \cap Z'$.
Let $i_Z: Z \setminus B_Z \to Z$ and $i_{Z'}: Z' \setminus B_{Z'} \to Z'$ be 
open immersions.
We replace the residue homomorphism by the composition of 
the following morphisms between objects on $X$:
\[
R(i_Z)_*\mathbf{Q}_{Z \setminus B_Z}
\to R(i_{Z'})_*R\Gamma_{Z' \setminus B_{Z'}}
(\mathbf{Q}_{Z \setminus B'_Z})[1] 
\cong R(i_{Z'})_*\mathbf{Q}_{Z' \setminus B_{Z'}}[-1].
\]
By taking the alternating sum, we obtain 
\[
R(i_q)_*\mathbf{Q}_{X^{[q]} \setminus B^{[q]}} 
\to R(i_{q+1})_*\mathbf{Q}_{X^{[q+1]} \setminus B^{[q+1]}}[-1]
\]
We have the following substitute of the Mayer-Vietoris exact sequence:

\begin{Lem}
There exists an object $H_{\mathbf{Q}}$ on $X$
as a convolution of the following complex:
\[
R(i_0)_*\mathbf{Q}_{X^{[0]} \setminus B^{[0]}}
\to R(i_1)_*\mathbf{Q}_{X^{[1]} \setminus B^{[1]}}[-1]
\to R(i_2)_*\mathbf{Q}_{X^{[2]} \setminus B^{[2]}}[-2]
\to \dots
\]
where the direct images of objects by closed immersions to $X$ 
are denoted by the same symbols, such that 
\[
H_{\mathbf{Q}} \otimes \mathbf{C}
\cong \tilde{\Omega}^{\bullet}_X(\log B).
\]
\end{Lem}

\begin{proof}
The convolution of a complex of objects may not exist nor be unique in general.
In our case we have isomorphisms
\[
\alpha_q: R(i_q)_*\mathbf{Q}_{X^{[q]} \setminus B^{[q]}} \otimes \mathbf{C}
\cong \Omega^{\bullet}_{X^{[q]}}(\log B^{[q]}).
\]
which are compatible with the boundary morphisms of the both 
complexes of objects.

We execute the convolution process from the right.
Let 
\[
H_0 \to H_1 \to \dots \to H_{k-1} \to H_k
\]
be a complex at an intermediate step, then a complex of the 
next step is constructed in the following way.
If 
\[
H' \to H_{k-1} \to H_k \to H'[1]
\]
is a distinguished triangle, then the boundary morphism 
$H_{k-2} \to H_{k-1}$ is liftable to a morphism $H_{k-2} \to H'$ because 
its composition with the boundary morphism $H_{k-1} \to H_k$ vanishes.
If the composition $H_{k-3} \to H_{k-2} \to H'$ is zero, then we obtain a
complex of the next step with shorter length.

If we tensorize everything in the above procedure with $\mathbf{C}$, 
then the corresponding complex at the intermediate step is given by 
\[
\begin{split}
&\Omega^{\bullet}_{X^{[0]}}(\log B^{[0]})
\to \Omega^{\bullet}_{X^{[1]}}(\log B^{[1]})[-1] 
\to \dots \to \Omega^{\bullet}_{X^{[k-1]}}(\log B^{[k-1]})[-k+1] \\ 
&\to \text{Ker}\{\Omega^{\bullet}_{X^{[k]}}(\log B^{[k]})[-k] 
\to \Omega^{\bullet}_{X^{[k+1]}}(\log B^{[k+1]})[-k-1]\}. 
\end{split}
\]
Then $H_{k-2} \otimes \mathbf{C}$ and $H' \otimes \mathbf{C}$ are 
respectively isomorphic to 
\[
\begin{split}
&\Omega^{\bullet}_{X^{[k-2]}}(\log B^{[k-2]})[-k+2] \\
&\text{Ker}\{\Omega^{\bullet}_{X^{[k-1]}}(\log B^{[k-1]})[-k+1] 
\to \Omega^{\bullet}_{X^{[k]}}(\log B^{[k]})[-k]\}.
\end{split}
\]
Therefore the lifted morphism 
$H_{k-2} \otimes \mathbf{C} \to H' \otimes \mathbf{C}$
is uniquely determined by the compatibility with the given boundary morphism
$\Omega^{\bullet}_{X^{[k-2]}}(\log B^{[k-2]})[-k+2]
\to \Omega^{\bullet}_{X^{[k-1]}}(\log B^{[k-1]})[-k+1]$. 
It follows that the lifted morphism 
$H_{k-2} \to H'$ is also unique and gives a shortened complex of the next step
yielding the desired convolution.
\end{proof}

For example, if $X$ is embedded into a 
smooth variety $V$ with a simple normal crossing divisor $X \cup W$ such that 
$B = X \cap W$, then we have
\[
H_{\mathbf{Q}} \cong R\Gamma_XR(i_V)_*\mathbf{Q}_{V \setminus (X \cup W)}[2]
\]
where $i_V: V \setminus (X \cup W) \to V$ is the open immersion.

We define a {\em weight filtration}, an increasing filtration denoted by $W$,
on the De Rham complex
$\tilde{\Omega}^{\bullet}_X(\log B)$ by the following exact sequence
\[
\begin{split}
&0 \to W_q(\tilde{\Omega}^{\bullet}_X(\log B))
\to W_q(\Omega^{\bullet}_{X^{[0]}}(\log B^{[0]})) \\
&\to W_{q-1}(\Omega^{\bullet}_{X^{[1]}}(\log B^{[1]})[-1]) 
\to W_{q-2}(\Omega^{\bullet}_{X^{[2]}}(\log B^{[2]})[-2])
\to \dots 
\end{split}
\]
where the $W$'s from the second terms denote 
the filtration with respect to the order of log poles.
We define a {\em Hodge filtration}, a decreasing filtration denoted by $F$, by
\[
F^p(\tilde{\Omega}^{\bullet}_X(\log B))
= \tilde{\Omega}^{\ge p}_X(\log B).
\]

\begin{Lem}\label{Gr}
\[
\begin{split}
&\text{Gr}_q^W(\tilde{\Omega}^{\bullet}_X(\log B))
\cong \bigoplus_{dim X - \dim Z = q} \Omega_Z^{\bullet}[-q] \\
&F^r(\text{Gr}_q^W(\tilde{\Omega}^{\bullet}_X(\log B)))
\cong \bigoplus_{dim X - \dim Z = q} \Omega_Z^{\ge r-q}[-q]
\end{split}
\]
where the $Z$ run all the 
closed strata of the pair $(X,B)$ of codimension $q$.
\end{Lem}

\begin{proof}
Let $X^{[p],[q]}$ be the disjoint union of all the closed strata of the pair 
$(X,B)$ which are contained in $X^{[p]}$ and has codimension $q$ 
in $X^{[p]}$.
Then the residue homomorphism gives an isomorphism
\[
\text{Gr}_{q-p}^W(\Omega_{X^{[p]}}^{\bullet}(\log B^{[p]})) 
\cong \Omega_{X^{[p],[q-p]}}^{\bullet}[-q+p].
\]
Thus we obtain exact sequences
\[
0 \to \text{Gr}_q^W(\tilde{\Omega}^{\bullet}_X(\log B))[q] 
\to \Omega_{X^{[0],[q]}}^{\bullet}
\to \Omega_{X^{[1],[q-1]}}^{\bullet}
\to \Omega_{X^{[2],[q-2]}}^{\bullet}
\to \dots 
\]
Let $Z$ be an irreducible component of $X^{[p],[q-p]}$ 
but not of $X^{[p+1],[q-p-1]}$ for some $p \le q$.
Then $Z$ appears $\binom{p+1}{p'+1}$ times in $X^{[p'],[q-p']}$ for 
$0 \le p' \le p$.
Therefore the contributions of $Z$ cancel except once.
\end{proof}

We define a {\em weight filtration} $W_q(H_{\mathbf{Q}})$ on the 
$\mathbf{Q}$-level object $H_{\mathbf{Q}}$
as a convolution of the following complex
\[
\begin{split}
&\tau_{\le q}(R(i_0)_*\mathbf{Q}_{X^{[0]} \setminus B^{[0]}})
\to \tau_{\le q}(R(i_1)_*\mathbf{Q}_{X^{[1]} \setminus B^{[1]}}[-1]) \\
&\to \tau_{\le q}(R(i_2)_*\mathbf{Q}_{X^{[2]} \setminus B^{[2]}}[-2])
\to \dots
\end{split}
\]
where $\tau$ denotes the canonical filtration
such that there is an isomorphism
\[
W_q(H_{\mathbf{Q}}) \otimes \mathbf{C} \cong 
W_q(\tilde{\Omega}^{\bullet}_X(\log B)).
\]
The existence of such a convolution is guaranteed by the same reason as before.
Here we note that 
\[
\tau_{\le q}(R(i_p)_*\mathbf{Q}_{X^{[p]} \setminus B^{[p]}}[-p])
= \tau_{\le q-p}(R(i_p)_*\mathbf{Q}_{X^{[p]} \setminus B^{[p]}})[-p].
\]

The \lq\lq filtration'' of the object $H_{\mathbf{Q}}$ means the following: 
$H_{\mathbf{Q}}$ is a convolution of a complex
\[
\to \text{Gr}^W_{q+1}(H_{\mathbf{Q}})[-q-1] \to 
\text{Gr}^W_q(H_{\mathbf{Q}})[-q] 
\to \text{Gr}^W_{q-1}(H_{\mathbf{Q}})[-q+1] \to \dots
\]
where the objects $W_q(H_{\mathbf{Q}})$ and 
$\text{Gr}^W_q(H_{\mathbf{Q}})$ appear in the intermediate steps of the
convolution process in the following distinguished triangles
\[
W_{q-1}(H_{\mathbf{Q}}) \to W_q(H_{\mathbf{Q}}) 
\to \text{Gr}^W_q(H_{\mathbf{Q}}) \to W_{q-1}(H_{\mathbf{Q}})[1].
\]
By the same argument as before, we obtain 
\[
\text{Gr}_q^W(H_{\mathbf{Q}})
\cong \bigoplus_{dim X - \dim Z = q} \mathbf{Q}_Z[-q].
\]
Hence we have the following:

\begin{Thm}
\[
((H_{\mathbf{Q}}, W), 
(\tilde{\Omega}^{\bullet}_X(\log B), W, F))
\]
is a cohomological mixed Hodge $\mathbf{Q}$-complex on $X$.
\end{Thm}

\begin{proof}
If we define a shifted filtration $F(-q)$ on $\Omega_Z^{\bullet}$
by $F(-q)^r = F^{r-q}$, then 
\[
(\mathbf{Q}_Z,\Omega_Z^{\bullet},F(-q))
\]
is a cohomological Hodge $\mathbf{Q}$-complex of weight $2q$.
Hence
\[
(\mathbf{Q}_Z[-q],\Omega_Z^{\bullet}[-q],F(-q)[-q])
\]
is a cohomological Hodge $\mathbf{Q}$-complex of weight $q$.
\end{proof}

\begin{Cor}
The weight spectral sequence
\[
{}_WE_1^{p,q} = H^{p+q}(X,\text{Gr}^W_{-p}(H_{\mathbf{Q}}))
\Rightarrow H^{p+q}(H_{\mathbf{Q}})
\]
degenerates at $E_2$, the Hodge spectral sequence
\[
{}_FE_1^{p,q} = H^q(X,(\tilde{\Omega}^p_X(\log B))) \Rightarrow
H^{p+q}(\tilde{\Omega}^{\bullet}_X(\log B))
\]
degenerates at $E_1$, and the cohomology group
$H^n(H_{\mathbf{Q}})$ underlies a mixed Hodge $\mathbf{Q}$-structure
for any $n \ge 0$.
\end{Cor}

\begin{Expl}
Assume that $X = C_1 \cup C_2$, $C_1 \cong C_2 \cong \mathbf{P}^1$, 
$C_1 \cap C_2 = \{P_1,P_2\}$, and $B=0$.
Then
\[
\begin{split}
&\text{Gr}_F^0(\tilde{\Omega}_X^{\bullet})
= \tilde{\Omega}_X^0 
\cong \mathcal{O}_{C_1} \oplus \mathcal{O}_{C_2} \\
&\text{Gr}_F^1(\tilde{\Omega}_X^{\bullet}) 
= \tilde{\Omega}_X^1[-1] \\
&= \text{Ker}\{\Omega^1_{C_1}(\log (P_1+P_2)) 
\oplus \Omega^1_{C_2}(\log (P_1+P_2))
\to \mathcal{O}_{P_1} \oplus \mathcal{O}_{P_2}\}[-1] \\
&\cong \omega_X[-1]
\end{split}
\]
and 
\[
\begin{split}
&\text{Gr}^W_0(\tilde{\Omega}_X^{\bullet}) 
= \{\mathcal{O}_{C_1} \oplus \mathcal{O}_{C_2} 
\to \Omega^1_{C_1} \oplus \Omega^1_{C_2}\} \\
&\text{Gr}^W_1(\tilde{\Omega}_X^{\bullet}) 
= \text{Ker}\{(\mathcal{O}_{P_1}^2 \oplus \mathcal{O}_{P_2}^2)[-1]
\to (\mathcal{O}_{P_1} \oplus \mathcal{O}_{P_2})[-1]\} \\
&\cong (\mathcal{O}_{P_1} \oplus \mathcal{O}_{P_2})[-1]
\end{split}
\]
Therefore we have the following non-zero cohomology groups:
\[
\dim H^0(\text{Gr}_F^0(\tilde{\Omega}_X^{\bullet})) = 2,\,
\dim H^1(\text{Gr}_F^1(\tilde{\Omega}_X^{\bullet})) = 1,\,
\dim H^2(\text{Gr}_F^1(\tilde{\Omega}_X^{\bullet})) = 1
\]
and 
\[
\dim H^0(\text{Gr}^W_0(\tilde{\Omega}_X^{\bullet})) = 2,\,
\dim H^1(\text{Gr}^W_1(\tilde{\Omega}_X^{\bullet})) = 2,\,
\dim H^2(\text{Gr}^W_0(\tilde{\Omega}_X^{\bullet})) = 2.
\]
\end{Expl}


\section{Relative case}

We shall extend the construction of the previous section to the 
relative setting.

We consider the situation of the main theorem.
We set $Y^o = Y \setminus C$, $X^o = f^{-1}(Y^o)$, 
$f^o = f \vert_{X^o}$ and $B^o = B \cap X^o$.
Thus $f^o$ is a smooth and projective morphism.
We set $X^{o[q]} = X^{[q]} \cap X^o$ and $B^{o[q]} = B^{[q]} \cap X^o$ with 
open immersions $i_q: X^{o[q]} \setminus B^{o[q]} \to X^{o[q]}$.

We define the relative De Rham complex 
\[
\tilde{\Omega}^{\bullet}_{X^o/Y^o}(\log B^o)
\]
to be the quotient of a differential graded algebra 
$\tilde{\Omega}^{\bullet}_{X^o}(\log B^o)$
by an ideal generated by a locally free subsheaf $f^*\Omega^1_{Y^o}$
of $\tilde{\Omega}^1_{X^o}(\log B^o)$.
Then 
\[
\tilde{\Omega}^{\dim X - \dim Y}_{X^o/Y^o}(\log B^o) 
\cong \omega_{X^o/Y^o}(B^o)
\]
is an invertible sheaf.
The residue exact sequence becomes
\[
\begin{split}
&0 \to \tilde{\Omega}^{\bullet}_{X^o/Y^o}(\log B^o)
\to \Omega^{\bullet}_{X^{o[0]}/Y^o}(\log B^{o[0]})
\to \Omega^{\bullet}_{X^{o[1]}/Y^o}(\log B^{o[1]})[-1] \\
&\to \Omega^{\bullet}_{X^{o[2]}/Y^o}(\log B^{o[2]})[-2]
\to \dots
\end{split}
\]
where the complexes $\Omega^{\bullet}_{X^{o[p]}/Y^o}(\log B^{o[p]})$ are 
defined similarly.

The {\em weight filtration} on the complex
$\tilde{\Omega}^{\bullet}_{X^o/Y^o}(\log B^o)$ is defined 
by the following exact sequence
\[
\begin{split}
&0 \to W_q(\tilde{\Omega}^{\bullet}_{X^o/Y^o}(\log B^o))
\to W_q(\Omega^{\bullet}_{X^{o[0]}/Y^o}(\log B^{o[0]})) \\
&\to W_{q-1}(\Omega^{\bullet}_{X^{o[1]}/Y}(\log B^{o[1]})[-1]) 
\to W_{q-2}(\Omega^{\bullet}_{X^{o[2]}/Y}(\log B^{o[2]})[-2])
\to \dots 
\end{split}
\]
where the $W$'s from the second terms denote 
the filtration with respect to the order of log poles,
and the {\em Hodge filtration} by
\[
F^p(\tilde{\Omega}^{\bullet}_{X^o/Y^o}(\log B^o))
= \tilde{\Omega}^{\ge p}_{X^o/Y^o}(\log B^o).
\]

Let $i_q^o: X^{o[q]} \setminus B^{o[q]} \to X^{o[q]}$ 
denote an open immersion for any $q$. 
Then we have 
\[
R(i_q^o)_*\mathbf{Q}_{X^{o[q]} \setminus B^{o[q]}} 
\otimes (f^o)^{-1}\mathcal{O}_{Y^o}
\cong \tilde{\Omega}^{\bullet}_{X^{o[q]}/Y^o}(\log B^{o[q]}).
\]
We deduce that there exists a convolution $H^o_{\mathbf{Q}}$ 
of the following complex of objects on $X^o$:
\[
\begin{split}
&R(i_0^o)_*\mathbf{Q}_{X^{o[0]} \setminus B^{o[0]}}
\to R(i_1^o)_*\mathbf{Q}_{X^{o[1]} \setminus B^{o[1]}}[-1] \\
&\to R(i_2^o)_*\mathbf{Q}_{X^{o[2]} \setminus B^{o[2]}}[-2] \to \dots
\end{split}
\]
such that 
\[
H^o_{\mathbf{Q}} \otimes (f^o)^{-1}\mathcal{O}_{Y^o}
\cong \tilde{\Omega}^{\bullet}_{X^o/Y^o}(\log B^o)
\]
as in the previous section.

Since our family $f^o$ is topologically locally trivial, 
the higher direct images
\[
R^pf^o_*R(i^o_q)_*\mathbf{Q}_{X^{o[q]} \setminus B^{o[q]}}
\]
are locally constant sheaves of $\mathbf{Q}$-modules on $Y^o$ for all 
$p$ and $q$, and the homomorphisms between them are flat.
It follows that the sheaf $R^nf^o_*H^o_{\mathbf{Q}}$ is also locally constant 
for any $n$.

In a similar way as before, we define a {\em weight filtration} 
$W_q(H_{\mathbf{Q}})$ on $H_{\mathbf{Q}}$ as convolutions 
of the following complexes: 
\[
\begin{split}
&\tau_{\le q}(R(i_0^o)_*\mathbf{Q}_{X^{o[0]} \setminus B^{o[0]}})
\to \tau_{\le q}(i_1^o)_*\mathbf{Q}_{X^{o[1]} \setminus B^{o[1]}}[-1]) \\
&\to \tau_{\le q}(i_2^o)_*\mathbf{Q}_{X^{o[2]} \setminus B^{o[2]}}[-2]) 
\to \dots
\end{split}
\]
such that 
\[
W_q(H^o_{\mathbf{Q}}) \otimes (f^o)^{-1}\mathcal{O}_{Y^o}
\cong W_q(\tilde{\Omega}^{\bullet}_{X^o/Y^o}(\log B^o)).
\]
The restrictions of above constructions to the fibers of 
$f$ coincide with those in the previous section.

The transversality of the natural connection $\nabla$ on 
$R^nf^o_*(\tilde{\Omega}^{\bullet}_{X^o/Y^o}(\log B^o))$
to the Hodge filtration follows from the topological local triviality of $f^o$.

Therefore we obtain the following:

\begin{Thm}
\[
((R^nf^o_*H^o_{\mathbf{Q}}, W), 
(R^nf^o_*(\tilde{\Omega}^{\bullet}_{X^o/Y^o}(\log B^o)), W, F))
\]
is a variation of mixed Hodge $\mathbf{Q}$-structures on $Y^o$
for any $n$.
\end{Thm}

\begin{Cor}\label{spec}
The weight spectral sequence
\[
{}_WE_1^{p,q} 
= R^{p+q}f_*\text{Gr}^W_{-p}(H_{\mathbf{Q}}) 
\Rightarrow R^{p+q}f_*(H_{\mathbf{Q}})
\]
degenerates at $E_2$, and the Hodge spectral sequence
\[
{}_FE_1^{p,q} = R^qf_*(\tilde{\Omega}^p_{X^o/Y^o}(\log B^o)) 
\Rightarrow R^{p+q}f_*f^{-1}\mathcal{O}_{Y^o}
\]
degenerates at $E_1$.
\end{Cor}


\section{Semipositivity theorem}

We define invertible sheaves $\omega_{X/Y}(B)$ 
and $\omega_{X^{[p]}/Y}(B^{[p]})$ on $X$ and $X^{[p]}$ rspectively by
\[
\omega_{X/Y}(B) = \omega_X(B) \otimes f^*\omega_Y, \quad
\omega_{X^{[p]}/Y}(B^{[p]}) = \omega_{X^{[p]}}(B^{[p]}) \otimes f^*\omega_Y.
\]
Then the residue homomorphisms yield an exact sequence of sheaves on 
$X$ (not only on $X^o$):
\[
0 \to \omega_{X/Y}(B) \to \omega_{X^{[0]}/Y}(B^{[0]})
\to \omega_{X^{[1]}/Y}(B^{[1]}) \to \omega_{X^{[2]}/Y}(B^{[2]})
\to \dots
\]

The semipositivity theorem (Theorem~\ref{main}) 
follows from the combination of the following theorem 
with a result in \cite{CA}:

\begin{Thm}
Let $\tilde{\mathcal{H}}^n$ be the canonical extension of the
variation of mixed Hodge structures
$\mathcal{H}^n = R^nf^o_*(\tilde{\Omega}^{\bullet}_{X^o/Y^o}(\log B^o))$,
and let $\tilde{F}$ be the extended Hodge filtration 
on $\tilde{\mathcal{H}}^n$.
Then there is an isomorphism
\[
R^qf_*\omega_{X/Y}(B) \cong F^{d(X/Y)}(\tilde{\mathcal{H}}^{q+d(X/Y)})
\]
where $d(X/Y) = \dim X - \dim Y$.
\end{Thm}

\begin{proof}
We put a weight filtration on $\omega_{X/Y}(B)$ by exact sequences
\[
\begin{split}
&0 \to W_q(\omega_{X/Y}(B)) \to W_q(\omega_{X^{[0]}/Y}(B^{[0]})) \\
&\to W_{q-1}(\omega_{X^{[1]}/Y}(B^{[1]})) 
\to W_{q-2}(\omega_{X^{[2]}/Y}(B^{[2]})) \to \dots
\end{split}
\]
where the $W$ from the second terms are filtrations with respect to the 
order of log poles.
It is the part of the highest Hodge filtrations of the weight filtration
in the previous section if we restrict the sheaves over $X^o$ and 
shift the degree by $-d(X/Y)$.

Let $X^{[p],[q]}$ be the disjoint 
union of all the closed strata of the pair $(X,B)$ 
which are contained in $X^{[p]}$ and has codimension $q$ in $X^{[p]}$.
Then we have 
\[
\begin{split}
&\text{Gr}^W_q(\omega_{X^{[p]}/Y}(B^{[p]})) 
\cong \omega_{X^{[p],[q]}/Y} \\
&\text{Gr}^W_q(\omega_{X/Y}(B)) 
\cong \bigoplus_{\dim X - \dim Z = q} \omega_{Z/Y}
\end{split}
\]
as before.

By Corollary~\ref{spec} and Lemma~\ref{Gr}, 
we have a weight spectral sequence
\begin{equation}\label{spec1}
{}_WE_1^{p,q} 
= \bigoplus_{dim X - \dim Z = -p} R^{2p+q}f^o_*\Omega_{Z^o/Y^o}^{\bullet}
\Rightarrow R^{p+q}f^o_*(\tilde{\Omega}^{\bullet}_{X^o/Y^o}(\log B^o))
\end{equation}
which degenerates at $E_2$, where we put $Z^o=Z \cap X^o$.
The boundary homomorphism $d_1^{p,q}$ of the spectral sequence (\ref{spec1}) 
at the $E_1$-level
is given by the sum of the connecting homomorphisms
\begin{equation}\label{conn1}
R^{2p+q}f^o_*\Omega^{\bullet}_{Z^o/Y^o} 
\to R^{2p+q+2}f^o_*\Omega^{\bullet}_{(Z')^o/Y^o}
\end{equation}
of the exact sequences
\[
0 \to \Omega^{\bullet}_{(Z')^o/Y^o} \to \Omega^{\bullet}_{(Z')^o/Y^o}(\log Z^o) 
\to \Omega^{\bullet}_{Z^o/Y^o}[-1] \to 0
\]
for the inclusions $Z \subset Z'$ of the closed strata of the pair $(X,B)$ 
such that $\dim Z = \dim X + p$ and $\dim Z' = \dim X + p + 1$,
where we put $(Z')^o=Z' \cap X^o$.

Let $\tilde{\mathcal{H}}^{2p+q}_Z$ be the canonical extensions of 
the variations of Hodge structures
$\mathcal{H}^{2p+q}_Z = R^{2p+q}f^o_*\Omega^{\bullet}_{Z^o/Y^o}$ for 
the closed strata $Z$ of codimension $-p$.
Since the formation of canonical extensions is functorial, 
(\ref{spec1}) extends to another spectral sequence
\begin{equation}\label{spec1-can}
{}_W\tilde E_1^{p,q} 
= \bigoplus_{dim X - \dim Z = -p} \tilde{\mathcal{H}}^{2p+q}_Z
\Rightarrow \tilde{\mathcal{H}}^{p+q}
\end{equation}

It is known by \cite{Kollar} and \cite{Nakayama} (and \cite{CA} if $q=0$) 
that the highest Hodge parts of the canoical extensions are 
isomorphic to the higher direct images of dualizing sheaves:
\[
F^{p+d(X/Y)}(\tilde{\mathcal{H}}^{2p+q}_Z) 
\cong R^{p+q-d(X/Y)}f_*\omega_{Z/Y}.
\]
In particular, the right hand side are locally free sheaves on $Y$.

On the other hand, we have another spectral sequence with respect 
to the weight filtration
\begin{equation}\label{spec2}
{}_W\bar E_1^{p,q} 
= \bigoplus_{\dim X - \dim Z = -p} R^{p+q}f_*\omega_{Z/Y}
\Rightarrow R^{p+q}f_*(\omega_{X/Y}(B))
\end{equation}
so that
\[
{}_W\bar E_1^{p,q-d(X/Y)}
\cong F^{p+d(X/Y)}({}_W\tilde E_1^{p,q}).
\]
The boundary homomorphism $\bar d_1^{p,q}$ of (\ref{spec2})
is given by the sum of the connecting homomorphisms
\begin{equation}\label{conn2}
R^{p+q}f_*\omega_{Z/Y} \to R^{p+q+1}f_*\omega_{Z'/Y}
\end{equation}
of the adjunction sequences
\[
0 \to \omega_{Z'/Y} \to \omega_{Z'/Y}(Z) \to \omega_{Z/Y} \to 0.
\]

The homomorphism (\ref{conn1}) is a morphism of variations of Hodge structures
of bidegree $(1,1)$.
If we take the parts of the highest Hodge fitrations in (\ref{conn1}), then 
we obtain a homomorphism
\[
R^{p+q-d(X/Y)}f^o_*\omega_{Z^o/Y^o} 
\to R^{p+q+1-d(X/Y)}f^o_*\omega_{(Z')^o/Y^o}
\]
which coincides with the restriction of the homomorphism (\ref{conn2}) 
to $Y^o$ if we shift the degree by $d(X/Y)$.
Therefore the $E_2$ terms of the spectral sequence 
(\ref{spec1}) are also variations of Hodge structures, and we have
\[
F^{p+d(X/Y)}({}_WE_2^{p,q}) \cong {}_W\bar E_2^{p,q-d(X/Y)} \vert_{Y^o}.
\]
It follows also that the canoical extensions of the ${}_WE_2^{p,q}$ coincide 
with the ${}_W\tilde E_2^{p,q}$, and 
\[
F^{p+d(X/Y)}({}_W\tilde E_2^{p,q}) \cong {}_W\bar E_2^{p,q-d(X/Y)}.
\]

We already know that the boundary homomorphisms $d_m^{p,q}$ 
of the spectral sequence (\ref{spec1})
vanish for $m \ge 2$ and all $p,q$.
Since the sheaves ${}_W\tilde E_2^{p,q}$ and ${}_W\bar E_2^{p,q}$ 
are locally free, it follows that $\tilde d_2^{p,q} = \bar d_2^{p,q} = 0$
for the corresponding boundary homomorphisms of the spectral sequences
(\ref{spec1-can}) and (\ref{spec2}), respectively,
hence $\tilde d_m^{p,q} = \bar d_m^{p,q} = 0$ for $m \ge 2$ and all $p,q$,
and we obtain the theorem.
\end{proof}

Let $L$ be the polarization of the total space 
$X$ given by the cohomology class of an ample line bundle.
We define the primitive parts of the cohomology groups thoughout the argument
as the kernels with respect to the operation taking the cup product with $L$.
Since this operation is compatible with all other operations, 
our variations of mixed $\mathbf{Q}$-Hodge structures 
become graded polarizable.
We note that the highest Hodge parts are not affected by this operation 
because they always vanish when multiplied by $L$.
By the semipositivity theorem of \cite{CA}, the sheaves 
$F^{p+d(X/Y)}({}_W\tilde E_2^{p,q})$ are numerically semipositive.
Since $R^qf_*\omega_{X/Y}(B)$ is an extension of these sheaves, 
it is also numerically semipositive.
This is the conclusion of the proof of Theorem~\ref{main}.


\end{document}